\documentclass{amsart}
\usepackage{amsmath}
\usepackage{amsfonts}
\usepackage{graphics}
\usepackage{epsfig}
\usepackage{amssymb}
\usepackage{amscd}
\usepackage[all]{xy}
\usepackage{latexsym}
\usepackage{graphicx}
\usepackage{multirow}
\usepackage{geometry}
\usepackage{color}

\newtheorem{thm}{Theorem}[section]
\newtheorem{cor}[thm]{Corollary}

\newtheorem{prop}[thm]{Proposition}

\theoremstyle{definition}
\newtheorem{Def}[thm]{Definition}
\newtheorem{rem}[thm]{Remark}
\newtheorem*{ack}{Acknowledgement}

\numberwithin{equation}{section}
\numberwithin{figure}{section}

\def\Hom{{\text{\rm{Hom}}}}

\def\rchi{{\hbox{\raise1.5pt\hbox{$\chi$}}}}
\def\Aut{{\text{\rm{Aut}}}}

\def\a{\alpha}
\def\b{\beta}

\def\gam{\gamma}
\def\Gam{\Gamma}

\newcommand{\M}{\overline{\MM}}       
\newcommand{\MM}{\mathcal{M}}          
\newcommand{\mgn}{\MM_{g,n}}
\newcommand{\mgnbar}{\M_{g,n}}

\newcommand{\cl}[1]{[\![{#1} ]\!]}

\newcommand{\Mbar}{{\overline{\mathcal{M}}}}

\newcommand{\bP}{{\mathbb{P}}}
\newcommand{\bC}{{\mathbb{C}}}

\newcommand{\bZ}{{\mathbb{Z}}}

\newcommand{\cM}{{\mathcal{M}}}
\newcommand{\cN}{{\mathcal{N}}}
\newcommand{\cC}{{\mathcal{C}}}
\newcommand{\cD}{{\mathcal{D}}}
\newcommand{\cF}{{\mathcal{F}}}

\newcommand{\cW}{{\mathcal{W}}}
\newcommand{\cw}{{\mathfrak{w}}}
\newcommand{\cX}{{\mathcal{X}}}
\newcommand{\la}{{\langle}}
\newcommand{\ra}{{\rangle}}
\newcommand{\half}{{\frac{1}{2}}}

\newcommand{\lrar}{\longrightarrow}

\begin{document}
\large
\setcounter{section}{0}

\title[Topological recursion and topological quantum field theory]
{Topological recursion, topological quantum field theory and Gromov-Witten invariants of BG}

\author[D.\ Hern\'andez Serrano]{Daniel Hern\'andez Serrano}
\address{
Department of Mathematics and IUFFYM\\
University of Salamanca\\
Salamanca 37008, Spain}
\email{dani@usal.es}

\begin{abstract}
The purpose of this paper is to give a twisted version of the Eynard-Orantin topological recursion of \cite{EO1} by a 2D Topological Quantum Field Theory. We shall define a kernel for a 2D TQFT and use an algebraic definition for a topological recursion to define how to twist a standard topological recursion by a 2D TQFT. The A-model side enumerative problem consists of counting cell graphs where in addition vertices are decorated by elements in a Frobenius algebra, and which are a twisted version of the generalized Catalan numbers of \cite{DMSS,MS}.  We show that the function which counts these decorated graphs, which is a twist of the counting function of \cite{DMSS} by a Frobenius algebra, satisfies the same type of recursion  of \cite{CMS} with respect to the edge-contraction axioms of \cite{OM3}. The path we follow to pass from the A-model side to the remodelled B-model side is to use a discrete Laplace transform as a mirror symmetry map, based on the ideas of \cite{EMS,MZ,CMS,MS, MP2012}. We show that a twisted version by a 2D TQFT of the Eynard-Orantin differentials satisfies a twisted generalization of the topological recursion formula of \cite{EO1}. We shall illustrate these results with a toy model for the theory arising from the orbifold cohomology of the classifying space of a finite group. In this example, the graphs are orbifold cell graphs (graphs drawn on an orbifold punctured Riemann surface) defined out of the moduli space $\Mbar_{g,n}(BG)$ of stable morphisms from twisted curves to the classifying space of a finite group $G$. In particular we show that the cotangent class intersection numbers on the moduli space $\Mbar_{g,n}(BG)$ satisfy a twisted Eynard-Orantin topological recursion and we derive an orbifold DVV equation as a consequence of it. This proves from a different perspective the main result of \cite{JK}, which states that the  $\psi$-class intersection numbers on  $\Mbar_{g,n}(BG)$ satisfy the Virasoro  constraint condition.
\end{abstract}

\subjclass[2010]{Primary: 14N35, 14N10, 57R56, 81T45;
Secondary: 05C30, 55N32, 53D37}

\keywords{Topological quantum field 
theory; topological recursion; Frobenius 
algebras; ribbon graphs; orbifold cohomology; Gromov-Witten invariants}

\maketitle

\allowdisplaybreaks

\section{Introduction}
\label{sect:intro}

The recent formalism of a topological recursion given by Eynard-Orantin in \cite{EO1} has been a rich and powerful theory to interconnect different areas of mathematics and physics lately. Many of the uses of their recursion formula are based on the remodeling conjecture of \cite{BKMP}, which proposes this theory as a tool to compute the open Gromov-Witten invariants of a Calabi-Yau threefold when using its mirror curve as the spectral curve of Eynard-Orantin. The ideas of \cite{EMS,MZ,CMS,MS, MP2012} had contributed to solve many of its applications and to accept the Laplace transform as a mirror symmetry map, in the sense that the Laplace transform of many enumerative problems on the A-model side satisfies the Eynard-Orantin recursion on the B-model side for a particular choice of the spectral curve. Many examples has been proved before the conjecture itself were solved in \cite{FLZ,EO3}, such as counting lattice points of $\mgn$ \cite{CMS,N1,N2,MP2012}, single and orbifold Hurwitz numbers \cite{EMS,MZ,BHLM}, the Weil-Petersson volume of $\mgnbar$ \cite{EO2,MS,LX2,Mir1,Mir2}, the generalised Catalan numbers \cite{DMSS}, the stationary Gromov-Witten theory of $\bP^1$ \cite{DOSS} or the case of topological vertex \cite{C,Zhou2} among others. 

Recently, a new set of axioms for a 2D TQFT are given in \cite{OM3} an proved to be equivalent to the standard TQFT rules. One of the key points of this approach is that they transform the classic TQFT rules into new rules which reflects a reduction by $1$ of the topological quantity $2g-2+n$, which is one of the conditions needed for a topological recursion to be satisfied. Thus, it is natural to wonder if a 2D TQFT can be included into the topological recursion  formalism of Eynard-Orantin. 
The answer we propose in this paper consist of giving an algebraic version of the Eynard-Orantin topological recursion twisted by a 2D TQFT and prove that the Laplace transform of a twisted version of the Catalan numbers by a Frobenius algebra satisfy this new twisted topological recursion. We shall also provide a toy model for the theory by giving an example based on the orbifold cohomology of the classifying space of a finite group as a Frobenius algebra.

The paper is organized as follows. In section \ref{s:FA} we review the definitions of Frobenius algebra, 2D TQFT and ECA axioms. Section \ref{s:TWTR} is devoted to give a twisted version of a topological recursion by a 2D TQFT. We shall define kernel and cokernel operators for a 2D TQFT by using the product and coproduct of a finite dimensional commutative Frobenius algebra. We introduce an algebraic definition for a topological recursion to extend this operators and define a twisting of a standard topological recursion by a 2D TQFT. We show in section \ref{s:TWCAT} that a twisted generalization of the Catalan numbers satisfies a topological recursion. The A-model side enumerative problem consists of counting cell graphs where in addition vertices are decorated by elements in a Frobenius algebra, which are a twisted version of the generalized Catalan numbers of \cite{DMSS,MS} by a 2D TQFT. ECA axioms of \cite{OM3} allow us to show that the function which counts these decorated graphs satisfies the same type of recursion of \cite{CMS}. The Laplace transform of this recursion is the twisted version of the Eynard-Orantin topological recursion by a 2D TQFT proposed in the previous section.
Section \ref{s:GWBG} relates these results with the Gromov-Witten invariants of the classifying space $BG$ of a finite group $G$. We give an example arising from the orbifold cohomology of $BG$, where the decorated cell graphs are graphs drawn on an orbifold punctured Riemann surface defined out of the moduli space $\Mbar_{g,n}(BG)$ of stable morphisms from twisted curves to the classifying space of a finite group $G$, and which are given as an orbifold generalization of Grothendieck designs d'enfants. We generalize the lattice point counting of \cite{MP2012} to this orbifold setting and by taking the Laplace transform of the resulting recursion equation we show that  a twisted topological recursion by the 2D TQFT given by the orbifold cohomology of $BG$ as Frobenius algebra is satisfied.  This provide us with an orbifold DVV equation, which shows from a different perspective the main result of \cite{JK}: the $\psi$-class intersection numbers on $\Mbar_{g,n}(BG)$ satisfy the Virasoro constraint condition. We conclude in section \ref{s:proof} with the proof of Theorem \ref{thm:main2}.

\section{Frobenius algebras, 2D TQFT and ECA axioms.}
\label{s:FA}

In this section we review some definitions. We address the reader to follow \cite{Kock} for the notion of Frobenius algebra and its relation with $2$-dimensional topological quantum field theory (2D TQFT), \cite{At} for the mathematical definition of TQFT and \cite{OM3} for the edge-contraction axioms on cell graphs.

Let $A$ be a commutative Frobenius algebra over a field $K$ and let us denote:
\begin{itemize}
\item The product: 
\begin{align*}
m\colon A \otimes A &\to A \\
(u,v)&\mapsto m(u,v)\overset{not}{=}u\cdot v
\end{align*}
\item The non-degenerate symmetric bilinear form:
\begin{align*}
\eta \colon A \otimes A &\to K\\
(u,v)&\mapsto \eta(u,v)
\end{align*}
\item The Frobenius form:
\begin{align*}
\epsilon \colon A & \to K\\
u &\mapsto \epsilon(u):=\eta(1,u)
\end{align*}
\end{itemize}
There exists a unique coassociative coproduct $\delta \colon A \to A\otimes A$ whose counit is the Frobenius form $\epsilon \colon A \to K$ and which satisfies the Frobenius relation $$\delta\circ m=(\delta \otimes 1)\circ (1\otimes m)$$
In order to define the coproduct, let us introduce the three-point function
\begin{align*}
\phi \colon A\otimes A \otimes A &\to K\\
(u,v,w)&\mapsto \phi(u,v,w):=\eta(u\cdot v,w)=\eta(u,v\cdot w) 
\end{align*}

In terms of a $K$-basis $\{e_{1},\dots ,e_{s}\}$ of $A$ and the standard formula $u=\sum_{a,b}\eta(v,e_{a})\eta^{ab}e_{b}$ we can then write:
$$u\cdot v=\sum_{a,b}\phi(u,v,e_{a})\eta^{ab}e_{b}$$
$$\delta(v)=\sum_{i,j,a,b}\phi(v,e_{i},e_{j})\eta^{ia}\eta^{jb}e_{a}\otimes e_{b}\,,$$
where $\eta_{ij}:=\eta(e_{i},e_{j})$ and $\eta=(\eta_{ij})_{i,j}$ is the associated symmetric matrix, whose inverse is denoted $\eta^{-1}=(\eta^{ij})_{i,j}$. 

The last interesting operator (since the genus of a Riemann surface shall be codified on it) is the handle operator 
\begin{equation}
\label{e:handleop}
h \colon A \overset{\delta}{\to}A\otimes A \overset{m}{\to }A
\end{equation}
The image of $1\in A$ is the Euler element $\bold e=(m\circ \delta)(1)$.

\begin{Def}\cite{At, Kock}
A 2D TQFT is a rule $F$ which associates to each closed oriented $1$-manifold $\Sigma$ a vector space $A=F(\Sigma)$, and to each oriented cobordism $M\colon \Sigma_{1}\mapsto \Sigma_{2}$ associates a linear map $F(M)\colon F(\Sigma_{1})\to F(\Sigma_{2})$. This rule must satisfy:
\begin{itemize}
\item Two equivalent cobordisms must have the same image.
\item The cylinder cobordism from $\Sigma$ to itself must be sent to the identity map of $F(\Sigma)$.
\item Given a decomposition $M=M'M''$ then $F(M)$ is the composition of the linear maps $F(M')$ and $F(M'')$.
\item Disjoint union goes to tensor product, for $1$-manifolds and also for cobordisms. 
\item The empty manifold must be sent to the ground field $K$.
\item Takes the symmetry to the symmetry.
\end{itemize}
\end{Def}

Let $\Sigma_{g,n}$ be an oriented surface of type $(g,n)$ with labeled boundary components by indices $1,\dots ,n$. Let $A=F(S^{1})$ and $\Omega_{g,n}:=F(\Sigma_{g,n})\colon A^{\otimes n}\to K$ the associated multilinear map. We shall denote the associated 2D TQFT to a Frobenius algebra $A$ by the tuple $(A,\eta,\{\Omega_{g,n}\in A^{\otimes n*}\})$.

\begin{rem}
Let $\mgnbar$ be the moduli space of stable genus $g$ curves with $n$ marked points. A TQFT can be thought of as a Cohomological Field Theory which takes values in $H^{0}(\mgnbar,K)=K$.
\end{rem}

In \cite{OM3} a new set of rules for a 2D TQFT are given in terms of edge-contraction operations on cell-graphs, and proved to be equivalent to the standard set of axioms for a 2D TQFT. For the shake of completeness, let us include here these set of axioms. 

Let $\Gamma_{g,n}$ be the set of connected cell graphs of type $(g,n)$ with labeled vertices. Recall that a cell graph of type $(g,n)$ is a $1$-skeleton of a cell-decomposition of a connected compact oriented topological surface of genus $g$ with $n$ labeled $0$-cells, where a $0$-cell is called a vertex, a $1$-cell an edge and a $2$-cell a face (see \cite{DMSS} for details). For each cell graph $\gamma \in \Gamma_{g,n}$ let 
\begin{align*}
\Omega(\gamma)\colon A^{\otimes n}&\to \bC\\
v_{1}\otimes \dots \otimes v_{n}&\mapsto \Omega(\gamma)(v_{1}, \dots ,v_{n})
\end{align*}
be a multilinear map which decorates the $i$-th vertex of $\gamma$ with an element $v_{i}\in A$.

\begin{Def}\cite[Definition 4.4]{OM3}
\label{def:ECA}
The \textbf{edge-contraction axioms}.

\begin{itemize}
\item \textbf{ECA 0}: For the  cell graph consisting of only one vertex without any edge, $\gam_0\in \Gam_{0,1}$, we define
\begin{equation}\label{ECA0}
\Omega(\gam_0)(v) = \epsilon(v), \qquad v\in A.
\end{equation}

\item \textbf{ECA 1}: Suppose there is an edge $E$ connecting the $i$-th vertex and the $j$-th vertex for $i<j$ in $\gam\in \Gam_{g,n}$. Let $\gam'\in \Gam_{g,n-1}$ denote the cell graph
obtained by contracting $E$. Then
\begin{equation}\label{ECA1}
\Omega(\gam)(v_1,\dots,v_n) = \Omega(\gam')(v_1,\dots,v_{i-1}, v_i\cdot v_j,v_{i+1},\dots, \widehat{v_j},\dots,v_n),
\end{equation}
Here $\widehat{v_j}$ means we omit the $j$-th variable $v_j$ at the $j$-th vertex of $\gam$.
\item \textbf{ECA 2}: Suppose there is a loop $L$ attached at the $i$-th vertex of $\gam\in\Gam_{g,n}$.
Let $\gam'$ denote the possibly  disconnected graph obtained by contracting $L$ and separating the vertex
to two distinct vertices labeled by $i$ and $i'$. We assign an ordering $i-1<i<i'<i+1$.

If $\gam'$ is connected, then it is in $\Gam_{g-1,n+1}$.
We then impose
\begin{equation}\label{ECA2-1}
\Omega(\gam)(v_1,\dots,v_n) = \Omega(\gam')(v_1,\dots,v_{i-1},\delta(v_i),v_{i+1},\dots,v_n),
\end{equation}
where the outcome of the comultiplication $\delta(v_i)$ is placed in the $i$-th and $i'$-th slots.

If $\gam'$ is disconnected, then write $\gam'=(\gam_1,\gam_2)\in \Gam_{g_1,|I|+1} \times \Gam_{g_2,|J|+1}$, where 
\begin{equation}\label{disconnected}
\begin{cases}
g=g_1+g_2\\
I\sqcup J = \{1,\dots,\widehat{i},\dots,n\}
\end{cases}.
\end{equation}
Here, vertices labeled by $I$ belong to the connected component of genus $g_1$, and those labeled by $J$ on the other component. Let $(I_-,i,I_+)$ (reps. $(J_-,i,J_+)$) be reordering of $I\sqcup \{i\}$ (resp. $J\sqcup \{i\}$) in the increasing order. We impose
\begin{equation}\label{ECA2-2}
\Omega(\gam)(v_1,\dots,v_n) = \sum_{a,b,k,\ell}\eta(v_i,e_ke_\ell)\eta^{ka}\eta^{\ell b}\Omega(\gam_1)(v_{I_-},e_a,v_{I_+})\Omega(\gam_2)(v_{J_-},e_b,v_{J_+}).
\end{equation}
\end{itemize}
\end{Def}

Theorem $3.8$ and Collorary 4.8 of \cite{OM3} prove that given a Frobenius algebra $A$, the standard axioms of 2D TQFT and the ECA axioms are equivalent, moreover they have:
$$\Omega_{g,n}(v_{1},\dots, v_{n})=\epsilon(v_{1}\cdots v_{n}\cdot \bold e^{g})=\Omega(\gamma)(v_{1},\dots, v_{n})\,,$$
where $\bold e^{g}$ is the $g$-th power of the Euler element.

\section{Twisted Topological Recursion by a 2D TQFT}\label{s:TWTR}

\subsection{Topological Recursion.}\label{ss:EOrec}\qquad 

We address the reader to look at the topological recursion of Eynard-Orantin in \cite{EO1}, and to the special case of genus $0$ spectral curve to the mathematical definition given in \cite{DMSS}. In this subsection instead, we shall give an algebraic approach which shall be used to define how to twist a topological recursion by a 2D TQFT.

Let $\Sigma$ be a spectral curve and let us denote $V=H^{0}\big(\Sigma,K_{\Sigma}(*R)\big)$ the space of meromorphic differentials on $\Sigma$ (where $R$ is the set of ramification points of the spectral curve) and let us write $V^{n}=Sym^{n}H^{0}\big(\Sigma,K_{\Sigma}(*R)\big)$. 

\begin{Def}\label{d:TRKernel}
We define a topological recursion \emph{kernel operator} as the following map:
\begin{align*}
K\colon V \otimes V &\to V \\
(f_{0},f_{1}) & \mapsto K(f_{0},f_{1})\,,
\end{align*}
\end{Def}

which can naturally be extended to
\begin{align*}
K\colon V \otimes V \otimes V^{n-1}&\to V  \otimes V^{n-1}\\
(f_{0},f_{1},f_{2},\dots ,f_{n})&\mapsto \big(K(f_{0},f_{1}),f_{2},\dots ,f_{n}\big)
\end{align*}
\begin{align*}
K\colon V \otimes V^{|I|}\otimes V \otimes V^{|J|}&\to V  \otimes V^{|I\sqcup J|}\\
(f_{0},f_{I},f_{1},f_{J})&\mapsto \big(K(f_{0},f_{1}),f_{I},f_{J}\big)
\end{align*}

\begin{Def}\label{d:TR}
Let $(g,n)$ be a pair in the stable range. The meromorphic differentials $W_{g,n}\in V^{n}$ are said to satisfy a topological recursion (TR) w.r.t. the spectral curve $\Sigma$ and kernel $K$ if:
$$W_{g,n}=K(W_{g-1,n+1})+\frac{1}{2}\sum_{\substack{g_1+g_2=g\\
I\sqcup J=\{2,\dots,n\}}} ^{\text{no (0,1)}}K(W_{g_{1},|I|+1},W_{g_{2},|J|+1})\,.$$
It shall be called a Eynard-Orantin topological recursion (EO TR) when a explicit form of a EO kernel is chosen.
\end{Def}

\begin{rem}
The unstable differentials are also defined for a topological recursion by using the spectral curve, see \cite{EO1,DMSS} for details.
\end{rem}

\subsection{The kernel and the cokernel operators in a TQFT}\label{ss:CohFTKernel}\quad 

In this subsection we shall define ``kernel and cokernel operators'' in a 2D TQFT by using the coproduct and the product of the Frobenius algebra. This will be a useful tool later on to intrinsically define a 2DTQFT-twisted topological recursion. 

Let $(A,\eta,\Omega_{g,n}\in A^{\otimes n *} )$ be a 2D TQFT. Let us start with the coproduct in the Frobenius algebra and consider the composite $$A \overset{\delta}{\to }A_{1}\otimes A_{2}\overset{\tau}{\to}A_{2}\otimes A_{1}\,,$$
where $A_{i}=A$, $\delta$ is the coproduct and $\tau$ is the standard twist isomorphism operator interchanging the first and second factors. Consider the dual map:
$$(A_{2}\otimes A_{1})^{*}\simeq A_{1}^{*}\otimes A_{2}^{*}\to A^{*}$$
 which by abuse of notation we shall still denote $\delta^{*}\colon A^{*}\otimes A^{*}\to A^{*}$. 

\begin{Def}\label{d:coprodKernel}
We define the \emph{kernel operator} as $\delta^{*}\colon A^{*}\otimes A^{*}\to A^{*}$. Which can be naturally extended to:
$$\delta^{*}\colon A^{*}\otimes A^{*}\otimes (A^{*})^{n-1}\otimes H^{0}(\M_{g-1,n+1})\to A^{*}\otimes (A^{*})^{n-1}\otimes H^{\bullet}(\mgnbar)$$
$$\delta^{*}\colon A^{*}\otimes A^{*}\otimes (A^{*})^{|I|}\otimes (A^{*})^{|J|}\otimes H^{0}(\M_{g_{1},|I|+1})\otimes H^{0}(\M_{g_{2},|J|+1})\to A^{*}\otimes (A^{*})^{|I\sqcup J|}\otimes H^{0}(\mgnbar)\,,$$
and which we shall still denote by $\delta^{*}$. 
\end{Def}
\begin{rem}
Even if $H^{0}(\mgnbar,K)=K$, we want now to leave it in the definition in order to keep truck of the topological type and to show how this definition can be generalized to Cohomological Field Theory.
\end{rem}

In this fashion we have that the equation
$$\delta^{*}(\Omega_{g-1,n+1})=\Omega_{g,n}$$
is equivalent to {\bf ECA 2} axiom of equation (\ref{ECA2-1})
$$\Omega_{g,n}(v_{1},\dots ,v_{n})=\Omega_{g-1,n+1}(\delta(v_{1}),v_{[n]\setminus \{1\}})\,.$$

Similarly, 
$$\delta^{*}(\Omega_{g_{1},|I|+1},\Omega_{g_{2},|J|+1})=\Omega_{g,n}$$ 
produces {\bf ECA 2} axiom of equation (\ref{ECA2-2})
$$\Omega_{g,n}(v_{1},\dots ,v_{n})=\sum_{a,b,k,\ell}\phi(v_i,e_k,e_\ell)\eta^{ka}\eta^{\ell b}\Omega_{g_{1},|I|+1}(v_{I_-},e_a,v_{I_+})\Omega_{g_{2},|J|+1}(v_{J_-},e_b,v_{J_+})\,.$$

In an analogous way, we can start with the product $m\colon A \otimes A \to A$ and define a cokernel operator.
\begin{Def}\label{d:prodKernel}
We define the \emph{cokernel operator} as $m^{*}\colon A^{*} \to A^{*}\otimes  A^{*}$, naturally extended to:
$$m^{*}\colon A^{*}\otimes (A^{*})^{n-2}\otimes H^{0}(\M_{g,n-1})\to A^{*}\otimes A^{*}\otimes (A^{*})^{n-2}\otimes H^{0}(\mgnbar)\,.$$
\end{Def}

We have that 
$$m^{*}(\Omega_{g,n-1})=\Omega_{g,n}$$
is equivalent {\bf ECA 1} axiom of equation (\ref{ECA1})
$$\Omega_{g,n}(v_{1},\dots ,v_{n})=\Omega_{g,n-1}(v_{1}\cdot v_{j},v_{[n]\setminus \{1,j\}})\,.$$

Finally, let us finish this section by pointing out the relation between the kernel and the cokernel operators. 
In the Frobenius algebra $A$ we have the following identity:
$$m=(1\times \eta)\circ (\delta \times 1)$$
$$m\colon A \otimes A \overset{\delta \times 1}{\to}A \otimes A \otimes A \overset{1\times \eta}{\to} A$$

Let us define the $(0,2)$ unstable number by using the pairing $\eta$:
$$\Omega_{0,2}(v_1,v_2):=\eta(v_1,v_2)$$
Then we have that 
\begin{equation}\label{e:prodKcoprodK}
m_{*}(\Omega_{g,n-1})=\delta^{*}(\Omega_{g,n-1},\Omega_{0,2})\,.
\end{equation}
Therefore, and once the unstable $(0,2)$ case is consistently defined in the theory, we could just use the ``kernel operator''.

\subsection{Twisting a topological recursion by a 2D TQFT}\qquad

\begin{Def} We define the \emph{twisted Kernel}, $K_{\delta}$, as the following product of the TR kernel of Definition \ref{d:TRKernel} with the kernel of Definition \ref{d:coprodKernel}:
\begin{align*}
K_{\delta}\colon (V\otimes A^{*})\otimes (V\otimes A^{*})&\overset{K\times \delta^{*}}{\to} (V\otimes A^{*})\\
\big((f_{0},\Omega_{i}),(f_{1},\Omega_{j})\big)&\mapsto \big(K(f_{0},f_{1}),\delta^{*}(\Omega_{i},\Omega_{j})\big)\,.
\end{align*}
We extend the twisted kernel to:
$$(V\otimes A^{*})\otimes (V\otimes A^{*})\otimes (V^{n-1}\otimes (A^{*})^{n-1})\overset{K_{\delta}}{\to} (V\otimes A^{*})\otimes (V^{n-1}\otimes (A^{*})^{n-1})$$

$$(V\otimes A^{*})\otimes (V\otimes A^{*})\otimes (V^{|I|}\otimes (A^{*})^{|I|})\otimes (V^{|J|}\otimes (A^{*})^{|J|})\overset{K_{\delta}}{\to}(V\otimes A^{*})\otimes (V^{n-1}\otimes (A^{*})^{n-1})$$

\end{Def}

\begin{Def}\label{d:twEODiff}
Let $(g,n)$ be a stable pair. We \emph{define} the \emph{twisted-meromorphic differentials} as elements $\cW_{g,n}\in V^{n}\otimes (A^{*})^{n}$. 
\end{Def}

\begin{Def}\label{d:twTR}
Let $(g,n)$ be a pair in the stable range. The twisted meromorphic differentials $\cW_{g,n}\in (V\otimes A^{*})^{n}$ are said to satisfy a topological recursion w.r.t. the spectral curve $\Sigma$ and twisted kernel $K_\delta=K\times \delta^*$ if:
$$\cW_{g,n}=K_{\delta}(\cW_{g-1,n+1})+\frac{1}{2}\sum_{\substack{g_1+g_2=g\\
I\sqcup J=\{2,\dots,n\}}} ^{\text{no (0,1)}}K_{\delta}(\cW_{g_{1},|I|+1},\cW_{g_{2},|J|+1})\,.$$
It shall be called a Eynard-Orantin twisted topological recursion when a explicit form of a EO kernel for $K$ is chosen.
\end{Def}

Let us point out that there is no $(0,1)$ terms on the last summand, but there are $(0,2)$ appearing as:
$$K_{\delta}(\cW_{g,n-1},\cW_{0,2})\,.$$
Using the cokernel $m^{*}$ of Definition \ref{d:prodKernel}, we can define a \emph{twisted cokernel} by $K_{*}:=K\times m^{*}$. Thus, by equation (\ref{e:prodKcoprodK}), and once we define $\cW_{0,2}(f_{1},f_{2};v_{1},v_{2}):=W_{0,2}(f_{1},f_{2})\Omega_{0,2}(v_{1},v_{2})$, we have that:
$$K_{*}(\cW_{g,n-1})=K_{\delta}(\cW_{g,n-1},\cW_{0,2})\,,$$
and thus, the TR could be also written as:
$$\cW_{g,n}=K_{*}(\cW_{g,n-1})+K_{\delta}(\cW_{g-1,n+1})+\frac{1}{2}\sum_{\substack{g_1+g_2=g\\
I\sqcup J=\{2,\dots,n\}}} ^{\text{stable}}K_{\delta}(\cW_{g_{1},|I|+1},\cW_{g_{2},|J|+1})\,.$$

\begin{rem}
These definitions can be naturally extended to Cohomological Field Theories and related with \cite{ABO,DOSS}. Properties and details shall be studied elsewhere. 
\end{rem}

\section{Twisted Topological Recursion for twisted Catalan numbers.}
\label{s:TWCAT}

\subsection{Background: Designs d'enfants}\quad

A \emph{dessins d'enfant} of type $(g,n)$ is a topological graph drawn on a genus $g$ connected smooth algebraic curve $C$ which is defined as the inverse image $b^{-1}([0,1])$ of the closed interval $[0,1]\subset\bP^1$ by a clean Belyi map $b:C\lrar \bP^1$ (a meromorphic morphism ramified at three points $\{0,1,\infty\}$ such that $n$ is the number of poles of $b$ without counting the multiplicity, the ramification type of $b$ above $1\in \bP^1$ is $(2,2,\dots,2)$). They are a special kind of metric ribbon graphs and the enumeration of clean Belyi  morphism is equivalent to the enumeration of certain ribbon graphs.

A \emph{dessin} is defined as the dual graph $\gamma = b^{-1}([1,i\infty])$,  where $[1,i\infty]= \{1+iy\;|\;0\le y\le \infty\}\subset \bP^1$. It has  $n$ labeled vertices and it is a connected cell graph. The number of
dessins with the automorphism factor is defined by
\begin{equation}
\label{eq:dessin count}
D_{g,n}(\mu_1,\dots,\mu_n) = \sum_{\substack{\gamma \text{ dessin of}\\
\text{type } (g,n)}} \frac{1}{|\Aut_D(\gamma)|},
\end{equation}
where $(\mu_1,\dots,\mu_n)$ are the prescribed degrees of the $n$ labeled vertices and $\Aut_D
(\gamma)$ is the automorphism of $\gamma$ preserving each vertex point-wise (see \cite{MP1998,DMSS} for details).

The \emph{generalized Catalan numbers} of type $(g,n)$ are defined in \cite{MSu} by:
$$C_{g,n}(\mu_{1},\dots ,\mu_{n})=\mu_{1}\cdots \mu_{n}D_{g,n}(\mu_{1},\dots ,\mu_{n})\,,$$ 
and they count dual graphs as before where in addition an outgoing arrow is placed on one of its incident half-edges to the vertices (this is done in order to kill the automorphisms). They are called arrowed cell graphs in \cite{OM3}.

\subsection{Counting decorated dessins.}\label{ss:decdessins}\quad 

We are interested in counting dessins where its vertices are decorated by elements in a Frobenius algebra $A$. We shall refer to them as \emph{decorated dessins}. 

Given a Frobenius algebra $A$, let 
$$\cD(\mu_{1},\dots ,\mu_{n}) \colon A^{\otimes n} \to K$$
the multilinear map which to each tuple $(v_{1},\dots ,v_{n})\in A^{\otimes n}$ associates the number $$\cD(\mu_{1},\dots ,\mu_{n};v_{1},\dots ,v_{n})$$ of \emph{decorated dessins} of type $(g,n)$ with prescribed vertices degree profile $(\mu_{1},\dots ,\mu_{n})$. Since the decorating procedure is graph independent, this number is
$$\cD(\mu_{1},\dots ,\mu_{n};v_{1},\dots ,v_{n})=D(\mu_{1},\dots ,\mu_{n})\Omega_{g,n}(v_{1},\dots ,v_{n})$$
where $(A,\eta,\Omega_{g,n})$ is the associated 2D TQFT to $A$. For the unstable case $(0,2)$, we define

\begin{equation}\label{e:DG02}
\cD_{0,2}(\mu_{1},\mu_{2};v_{1},v_{2}):=D_{0,2}(\mu_{1},\mu_{2})\Omega_{0,2}(v_{1},v_{2})
\end{equation}
where $\Omega_{0,2}(v_{1},v_{2}):=\eta(v_{1},v_{2})$ and $D_{0,2}(\mu_{1},\mu_{2})$ is given in \cite[Proposition 3.1]{DMSS}.

\begin{prop} The number of decorated dessins satisfies the following recursion equation: 

\begin{equation}\label{countdecdessins}
\begin{aligned}
&\cD_{g,n}(\mu_{1},\dots, \mu_{n};v_{1}, \dots ,v_{n})
\\
&=
\sum_{j=2} ^n \mu_j 
\cD_{g,n-1}(\mu_1+\mu_j-2,\mu_2,\dots,
\widehat{\mu_{j}},\dots,\mu_n;v_{1}\cdot v_{j},v_{2},\dots,\widehat{v_{j}},\dots,v_{n})
\\
&+
\sum_{\a+\b=\mu_1-2}
\cD_{g-1,n+1}(\a,\b,\mu_2,\dots,\mu_n;\delta(v_{1}),v_{2},\dots,v_{n}\big)
\\
&+
\sum_{\a+\b=\mu_1-2}
\sum_{\substack{
g_1+g_2=g\\I\sqcup J=\{2,\dots,n\}}}
\delta^{*}\big ( \cD_{g_1,|I|+1}(\a,\mu_I;\_,v_{I}),\cD_{g_2,|J|+1}(\b,\mu_J;\_,v_{J})\big)(v_{1})
\end{aligned}
\end{equation}
where $\delta^{*}$ is given by Definition \ref{d:coprodKernel}:
$$\begin{aligned}
&\delta^{*}\big ( \cD_{g_1,|I|+1}(\a,\mu_I;\_,v_{I}),\cD_{g_2,|J|+1}(\b,\mu_J;\_,v_{J})\big)(v_{1})=\sum_{k,\ell,a,b}\phi(v_{1},e_{k},e_{\ell})\eta^{ka}\eta^{{\ell}b}
\\
&\qquad 
\times
\left(D_{g_1,|I|+1}(\a,\mu_I)\cdot
\Omega_{g_1,|I|+1}(e_{a},v_{I})
\right)
\left(D_{g_2,|J|+1}(\b,\mu_J)\cdot
\Omega_{g_2,|J|+1}(e_{b},v_{J})
\right).
\end{aligned}$$
\end{prop}

\begin{proof}
It follows from \cite[Theorem 3.3]{DMSS} by applying the ECA axioms of \cite{OM3} (see definition \ref{def:ECA}). When we contract an edge which connects the vertex $1$ and the vertex $j>1$ we need to apply ECA 1 axiom of equation (\ref{ECA1}); if we contract an edge which forms a loop attached to vertex $1$, we need to apply ECA 2 axiom of equation (\ref{ECA2-1}) if the resulting dessin is connected, and ECA 2 axiom of equation (\ref{ECA2-2}) if the resulting dessin is the disjoint union of two dessins.
\end{proof}

Let $\cC(\mu_{1},\dots ,\mu_{n}) \colon A^{\otimes n} \to K$ be the function which for each tuple $(v_{1},\dots ,v_{n})\in A^{\otimes n}$ produces the number $\cC_{g,n}(\mu_{1},\dots, \mu_{n};v_{1}, \dots ,v_{n})$ of decorated arrowed cell graphs. We shall refer to it as \emph{twisted generalized Catalan numbers}.
\begin{cor}\label{c:twcat} The twisted generalized Catalan numbers satisfies the following recursion equation: 
\begin{equation}\label{counting formula for BG}
\begin{aligned}
&\cC_{g,n}(\mu_{1},\dots, \mu_{n};v_{1}, \dots ,v_{n})
\\
&=
\sum_{j=2} ^n \mu_j 
\cC_{g,n-1}(\mu_1+\mu_j-2,\mu_2,\dots,
\widehat{\mu_{j}},\dots,\mu_n;v_{1}\cdot v_{j},v_{2},\dots,\widehat{v_{j}},\dots,v_{n})
\\
&+
\sum_{\a+\b=\mu_1-2}
\cC_{g-1,n+1}(\a,\b,\mu_2,\dots,\mu_n;\delta(v_{1}),v_{2},\dots,v_{n}\big)
\\
&+
\sum_{\a+\b=\mu_1-2}
\sum_{\substack{
g_1+g_2=g\\I\sqcup J=\{2,\dots,n\}}}
\delta^{*}\big ( \cC_{g_1,|I|+1}(\a,\mu_I;\_,v_{I}),\cC_{g_2,|J|+1}(\b,\mu_J;\_,v_{J})\big)(v_{1})
\end{aligned}
\end{equation}
\end{cor}

\subsection{Twisted Topological Recursion for twisted generalized Catalan numbers.}\quad

We shall apply the Laplace transform to the equation (\ref{countdecdessins}) using the method of \cite{DMSS,MS} in order to proof that the twisted differentials of Definition \ref{d:twEODiff} satisfy the twisted topological recursion of Definition \ref{d:twTR}.

Let $\mu=(\mu_{1},\dots, \mu_{n})$ and let $F^{D}_{g,n}(t_1,\dots,t_n)=\sum_{\mu\in \bZ_+ ^n}D_{g,n}(\mu)e^{-(w_1\mu_1+\cdots + w_n \mu_n)}$ be the Laplace transform of the Catalan numbers, where the relation between coordinates are:
$$z_{j}=\frac{t_{j}+1}{t_{j}-1}\,;\qquad e^{w_{j}}=\frac{t_{j}+1}{t_{j}-1}+\frac{t_{j}-1}{t_{j}+1}$$

The Laplace transform of the twisted Catalan numbers is given by
{\small$$\cF_{g,n}(t_1,\dots,t_n;v_{1},\dots ,v_{n})=\sum_{\mu\in \bZ_+ ^n}
\cD_{g,n}(\mu;v_{1},\dots,v_{n})\; e^{-(w_1\mu_1
+\cdots + w_n \mu_n)}=F_{g,n}^{D}(t_1,\dots,t_n)\Omega_{g,n}(v_{1},\dots,v_{n})$$}

Let $W_{g,n}^{D}(t_{1},\dots ,t_{n})=dt_{1}\cdots dt_{n}F^{D}_{g,n}(t_1,\dots,t_n)$ be the differential forms of \cite{DMSS} and let $x_{j}$ be variables defined by $x_{j}=e^{w_{j}}$ and write
$$W_{g,n}^{D}(t_{1},\dots ,t_{n})=w_{g,n}^{D}(t_{1},\dots ,t_{n})dt_{1}\cdots dt_{n}=w_{g,n}(x_{1},\dots ,x_{n})dx_{1}\cdots dx_{n}$$

We have
$$\cw_{g,n}(x_{1},\dots ,x_{n};v_{1},\dots ,v_{n})=w_{g,n}(x_{1},\dots ,x_{n})\Omega_{g,n}(v_{1},\dots ,v_{n})$$
and the twisted meromorphic differentials
\begin{align*}
\cW_{g,n}(t_{1},\dots ,t_{n};v_{1},\dots ,v_{n})&=dt_{1}\cdots dt_{n}\cF_{g,n}(t_1,\dots,t_n;v_{1},\dots ,v_{n})=\\
&=W_{g,n}^{D}(t_{1},\dots ,t_{n})\Omega_{g,n}(v_{1},\dots ,v_{n})\,.
\end{align*}

\begin{prop}\label{p:wG}
The Laplace transform of the recursion formula (\ref{countdecdessins}) is the following ECA based
differential recursion:
\begin{multline}
\label{eq:LT of Dgn}
-x_1\;\cw_{g,n}(x_1,\dots,x_n;v_{1},\dots ,v_{n})=
\\
=
\sum_{j=2}^n
 \frac{\partial}{\partial x_j}
     \left(
    \frac{1}{x_j-x_1}
    \left(\cw_{g,n-1}(x_2,\dots,x_n;v_{1}\cdot v_{j},v_{[n]\setminus \{1,j\}})
    -
    \cw_{g,n-1}(x_{[n]\setminus \{j\}};v_{1}\cdot v_{j},v_{[n]\setminus \{1,j\}})
    \right)
    \right)
    \\
  +
 \cw_{g-1,n+1}(x_1,x_1,x_{[n]\setminus \{1\}};\delta(v_{1}),v_{[n]\setminus \{1\}})
 + \\
+ \sum_{\substack{g_1+g_2=g\\
 I\sqcup J=\{2,\dots,n\}}}
\delta^{*}\big( \cw_{g_1,|I|+1}(x_1,x_I;-,v_{I}),
  \cw_{g_2,|J|+1}(x_1,x_J;-,v_{J})\big)(v_{1}).
\end{multline}
\end{prop}

Following the computations for the unstable case $(0,2)$ of \cite{DMSS}, we define
\begin{multline}
\label{eq:W02D}
\cW_{0,2} (t_1,t_2;v_{1},v_{2})=W_{0,2}^D (t_1,t_2)\Omega_{0,2}(v_{1},v_{2})
=d_1d_2 F_{0,2} ^D(t_1,t_2)\Omega_{0,2}(v_{1},v_{2})=\\
=\big(\frac{dt_1\cdot dt_2}{(t_1-t_2)^2}
-\frac{dx_1\cdot dx_2}{(x_1-x_2)^2}\big)\Omega_{0,2}(v_{1},v_{2})
=\Omega_{0,2}(v_{1},v_{2})\frac{dt_1\cdot dt_2}{(t_1+t_2)^2}.
\end{multline}

\begin{thm}
\label{thm:main2}
The twisted differential forms
\begin{equation}
\label{eq:WgnG and WgnD2}
\cW_{g,n}(t_1,\dots,t_n;v_{1},\dots ,v_{n})
=d_1\cdots d_n \cF_{g,n}(t_1,\dots,t_n;v_{1},\dots ,v_{n})
\end{equation}
satisfy the Eynard-Orantin twisted topological recursion 
\begin{multline}
\label{eq:GEO}
\cW_{g,n}(t_1,t_2,\dots,t_n;v_{1},\dots ,v_{n})
= \frac{1}{2\pi i} 
\int_{\phi} K^{D}(t,t_1)
\Bigg[
\cW_{g-1,n+1}(t,-t,t_2,\dots,t_n;\delta(v_{1}),v_{[n]\setminus \{1\}})\\
+
\sum^{\text{No $(0,1)$ terms}} _
{\substack{g_1+g_2=g\\I\sqcup J=\{2,3,\dots,n\}}}
\delta^{*}\big(\cW_{g_1,|I|+1}(t,t_I;-,v_{I}),\cW_{g_2,|J|+1}(-t,t_J;-,v_{J})\big)(v_{1})
\Bigg].
\end{multline}
with respect to the
spectral curve
\begin{equation}
\label{eq:Gspectral}
\begin{cases}
x = z+\frac{1}{z}\\
y = -z 
\end{cases}
\end{equation}
and the recursion kernel
\begin{multline}
\label{eq:Gkernel}
K^{D}(t,t_1) =
\half\; \frac{\int_t ^{-t} W^{D}_{0,2}(\cdot,t_1)}
{W^{D}_{0,1}(-t)-W^{D}_{0,1}(t)}=
\frac{1}{2} 
\left(
\frac{1}{t+t_1}+\frac{1}{t-t_1}
\right)
\frac{1}{32}\cdot
\frac{(t^2-1)^3}{t^2}\cdot \frac{1}{dt}\cdot dt_1.
\end{multline}

The integration is taken with respect to
a contour $\phi$ in the complex $t$-plane
consisting of two concentric circles centered around
$0$, with a positively oriented 
small inner circle of radius $\epsilon$
and a large negatively oriented circle of radius
$1/\epsilon$. This annulus should enclose all values
of $\pm t_i$, $i=1,\dots,n$.
\end{thm}

\begin{proof}
Given in the appendix.
\end{proof}

\begin{rem}
Equation (\ref{eq:GEO}) is the twisted topological recursion of Definition \ref{d:twTR} w.r.t. the spectral curve $\{x=z+\frac{1}{z},y=-z\}$ of \cite{DMSS} and the twisted kernel $K_{\delta}=K\times \delta^*$, where the explicit Eynard-Orintin kernel of equation (\ref{eq:Gkernel}) has been chosen. This shows that the twisted Eynard-Orantin differentials satisfies a twisted Eynard-Orantin Topological Recursion which splits as the product of the EO TR of \cite{DMSS} and a 2D TQFT $(A,\eta,\Omega_{g,n})$.
\end{rem}

\section{Gromov-Witten theory of $BG$ and orbifold DVV equation.}
\label{s:GWBG}

\subsection{Background.}
Following \cite{ACV, AV, CR, JK}, the moduli stack $\Mbar_{g,n}(BG)$ of stable maps from 
$n$-pointed twisted curves of genus $g$ to $BG$ is a smooth, proper, Deligne-Mumford stack of dimension $3g-3+n$. The forgetful morphism 
\begin{equation}
\label{eq:forget}
\varphi : \Mbar_{g,n}(BG) \lrar\Mbar_{g,n}
\end{equation}
is generically finite. In particular, its restriction to the smooth locus,
\begin{equation}
\label{eq:forgetsmooth}
\varphi : \cM_{g,n}(BG) \lrar\cM_{g,n},
\end{equation}
is a finite morphism with the  fiber 
\begin{equation}
\label{eq:homCopen}
\Hom\big(\pi_1(C\setminus\{p_1,\dots,p_n\}),G\big)
\big/\!\!\big/G
\end{equation}
at each $[C,\{p_1,\dots,p_n\}]\in \cM_{g,n}$, where the $G$ action on the space of homomorphisms
is via conjugation action.

Let $I BG$ be the inertia stack, which decomposes as $$I BG = \coprod_{[\![r]\!]} BG_{[\![r]\!]}=\coprod_{[\![r]\!]}[pt/C(r)]\,,$$
where $C(r)$ denotes the centralizer of $r\in G$ and the evaluation morphisms $$ev_i: \mgnbar(BG) \to I BG$$ allows to see that the stack $\M_{g,n}(BG)$ breaks up as the disjoint union of open and closed
substacks 
$$\M_{g,n}(BG) =
\coprod_{([\![r_1]\!],\dots,[\![r_n]\!])}
\M_{g,n}(BG,[\![r_1]\!],\dots,[\![r_n]\!]),$$ where
$\M_{g,n}(BG,[\![r_1]\!],\dots,[\![r_n]\!])=
ev_1^{-1}(BG_{[\![r_1 ]\!]}) \cap \dots \cap
ev_n^{-1}(BG_{[\![r_n ]\!]})$ and $[\![r_i]\!]$ denotes the conjugacy class of the element $r_i$ in $G$. The map $$\mgn(BG , [\![r_1]\!], \dots, [\![r_n]\!]) \to \cM_{g, n}$$ is a finite morphism of degree 
\begin{equation}\label{e:OmegaNumber}
\Omega^G_{g,n}(\bold r) =\frac{|\cX^G_{g} (\bold r)|}{|G|}\,,
\end{equation}
where $\bold r=([\![r_1]\!],\dots,[\![r_n]\!])$ and  $$\cX^G_{g} (\bold r):=\{(\alpha_1, \dots, \alpha_g, \beta_1, \dots,
\beta_g, \sigma_1, \dots, \sigma_n) |\textstyle
\prod^g_{i=1}[\alpha_i,\beta_i] =\prod^n_{j=1}\sigma_j, \ \sigma_j
\in [\![r_j]\!]\text{ for all } j \}.$$

Let
$$A :=H^*_{orb}(BG,\bC) := H^*({IBG,\bC}) =\bigoplus_{[\![r]\!]} \bC$$
be the orbifold cohomology of $BG$ as a vector space and, for each conjugacy class $[\![r]\!]$ in $G$, let $e_{[\![r]\!]}$ denote a $\bC$-basis of $A$. It is know that $A$ is a Frobenius algebra isomorphic to the center of the group algebra of $G$ where the non-degenerated bilinear form $\eta \colon A\otimes A \to \bC$ is given by

\begin{equation}\label{eq:metric}
\eta_{ij}:=\eta (e_{[\![r_i]\!]},e_{[\![r_j]\!]})=
\frac{1}{|C(r_i)|}\delta_{[\![r_i]\!][\![r^{-1}_j]\!]},
\end{equation}
and the multiplication (orbifold product) $m\colon A \otimes A \to A$ is given by:

\begin{equation}\label{e:product}
m(e_{[\![r_{i}]\!]}, e_{[\![r_{j}]\!]})=e_{[\![r_{i}]\!]}e_{[\![r_{j}]\!]} =\sum_{\substack{\sigma_{i},\sigma_{j}\\ \sigma_{i}\in \cl{a_{i}}\\ \sigma_{j}\in \cl{a_{j}}}}\frac{|C(\sigma_i \sigma_j)|}{|G|}
e_{[\![\sigma_i \sigma_j]\!]}\,.
\end{equation}

In \cite{JK} it is proven that the collection
\begin{align*}
\Omega_{g,n}^{G}\colon A^{\otimes n}& \to H^{*}(\M_{g,n},\bC)\\
e_{\cl{r_{1}}}\otimes \cdots \otimes e_{\cl{r_{n}}}&\mapsto \Omega_{g,n}^{G}(e_{\cl{r_{1}}}\otimes \cdots \otimes e_{\cl{r_{n}}})=\varphi_{*}\big(ev_{1}^{*}(e_{\cl{r_{1}}})\dots ev_{n}^{*}(e_{\cl{r_{n}}})\big)=\Omega_{g,n}^{G}(\bold r)
\end{align*}
is a Cohomological Field Theory. In fact it is a 2D TQFT since it takes values in $H^{0}(\M_{g,n},\bC)=\bC$. 

\subsection{Orbifold generalized Catalan numbers and twisted topological recursion}\quad

The $3$-point function $\phi \colon A \otimes A \otimes A \to \bC$ is defined by: 
\begin{equation}\label{e:3-pt}
\phi(e_{[\![r_{1}]\!]},e_{[\![r_{2}]\!]},e_{[\![r_{3}]\!]}):= \eta (e_{[\![r_{1}]\!]}e_{[\![r_{2}]\!]},e_{[\![r_{3}]\!]})= \eta (e_{[\![r_{1}]\!]},e_{[\![r_{2}]\!]}e_{[\![r_{3}]\!]})=\epsilon(e_{[\![r_{1}]\!]}e_{[\![r_{2}]\!]}e_{[\![r_{3}]\!]})\,,
\end{equation}
where $\epsilon \colon A \to \bC$ is the counit, and \cite[Proposition 3.1]{JK} shows that
\begin{equation}\label{e:3point}
\phi(e_{[\![r_{1}]\!]},e_{[\![r_{2}]\!]},e_{[\![r_{3}]\!]})=\Omega_{0,3}^{G}(e_{[\![r_{1}]\!]},e_{[\![r_{2}]\!]},e_{[\![r_{3}]\!]})\,.
\end{equation}
Using CohFT 3 axiom of \cite[Definition 3.1]{OM3} and the expression for $v\in A$
$$v=\sum_{\cl{a},\cl{b}}\eta(v,e_{\cl{a}})\eta^{\cl{a}\cl{b}}e_{\cl{b}}$$
we can see that the genus $0$ values of the collection $\{\Omega_{g,n}^{G}\}$ are given by
$$\Omega^{G}_{0,n}(v_{\cl{r_{1}}}, \cdots ,v_{\cl{r_{n}}})=\epsilon(v_{\cl{r_{1}}}\cdots v_{\cl{r_{n}}})$$

Thus, it follows from \cite[Theorem 3.8]{OM3} that
$$\Omega^{G}_{g,n}(v_{\cl{r_{1}}}, \cdots ,v_{\cl{r_{n}}})=\epsilon(v_{\cl{r_{1}}} \cdots v_{\cl{r_{n}}}\bold e^{g})$$
where $\bold e^{g}$ denotes the $g$-th power of the Euler element.

Let $\Gamma_{g,n}$ be the set of connected cell graphs of type $(g,n)$ with labeled vertices and for each cell graph $\gamma \in \Gamma_{g,n}$ let 
\begin{align*}
\Omega(\gamma)\colon A^{\otimes n}&\to \bC\\
v_{\cl{r_{1}}}\otimes \dots \otimes v_{\cl{r_{n}}}&\mapsto \Omega(\gamma)(v_{\cl{r_{1}}}, \dots ,v_{\cl{r_{n}}})
\end{align*}
be an $n$-variable function which assigns $v_{\cl{r_{i}}}\in A$ to the $i$-th vertex of $\gamma$. 
\begin{rem}
This decorating function consists of keeping truck of the orbifold information at each marked orbifold point of the twisted curve which maps to $BG$. 
\end{rem}

Provided that we define 
\begin{equation}\label{e:Omega01}
\Omega^{G}_{0,1}(v):=\epsilon(v)
\end{equation}
\begin{equation}\label{e:Omega02}
\Omega^{G}_{0,2}(v_{\cl{r_{1}}},v_{\cl{r_{2}}}):=\eta(v_{\cl{r_{1}}},v_{\cl{r_{2}}})
\end{equation}
and using  \cite[Theorem 4.7, Corollary 4.8]{OM3} we have:
\begin{prop}
For each cell graph $\gamma \in \Gamma_{g,n}$ define $$
\Omega^{G}_{g,n}(v_{\cl{r_{1}}}, \dots ,v_{\cl{r_{n}}})=\Omega(\gamma)(v_{\cl{r_{1}}}, \dots ,v_{\cl{r_{n}}})\,.$$
Then the collection $\{\Omega^{G}_{g,n}\}$ satisfies the edge-contraction axioms of Definition \ref{def:ECA}. 

As a consequence, $\Omega^{G}_{g,n}(v_{\cl{r_{1}}}, \dots ,v_{\cl{r_{n}}})$ is symmetric with respect to permutation indices and $\{\Omega^{G}_{g,n}\}$ is the 2D TQFT associated with the Frobenius algebra $A=H^{*}_{orb}(BG)$.

\end{prop}

\begin{rem}
Let us recall that in terms of the three-point function defined in equation (\ref{e:3point}), the multiplication and the comultiplication in $A$ can be written as
\begin{equation}\label{e:prod}
v_{\cl{r_{1}}}v_{\cl{r_{2}}}=\sum_{\cl{a},\cl{b}}\Omega_{0,3}^{G}(v_{\cl{r_{1}}},v_{\cl{r_{2}}},e_{\cl{a}})\eta^{ab}e_{\cl{b}}
\end{equation}
\begin{equation}\label{e:coprod}
\delta (v)=\sum_{\cl{r_i},\cl{r_j},\cl{a},\cl{b}}\Omega_{0,3}^{G}(v,e_{\cl{r_i}},e_{\cl{r_j}})\eta^{ia}\eta^{jb}e_{\cl{a}}\otimes e_{\cl{b}}
\end{equation}

So that the product and coproduct of the Frobenius algebra can be thought of as the two kinds of orbifold pair of pants of Figure \ref{fig:prodPants}.
\begin{figure}[h!]
\epsfig{file=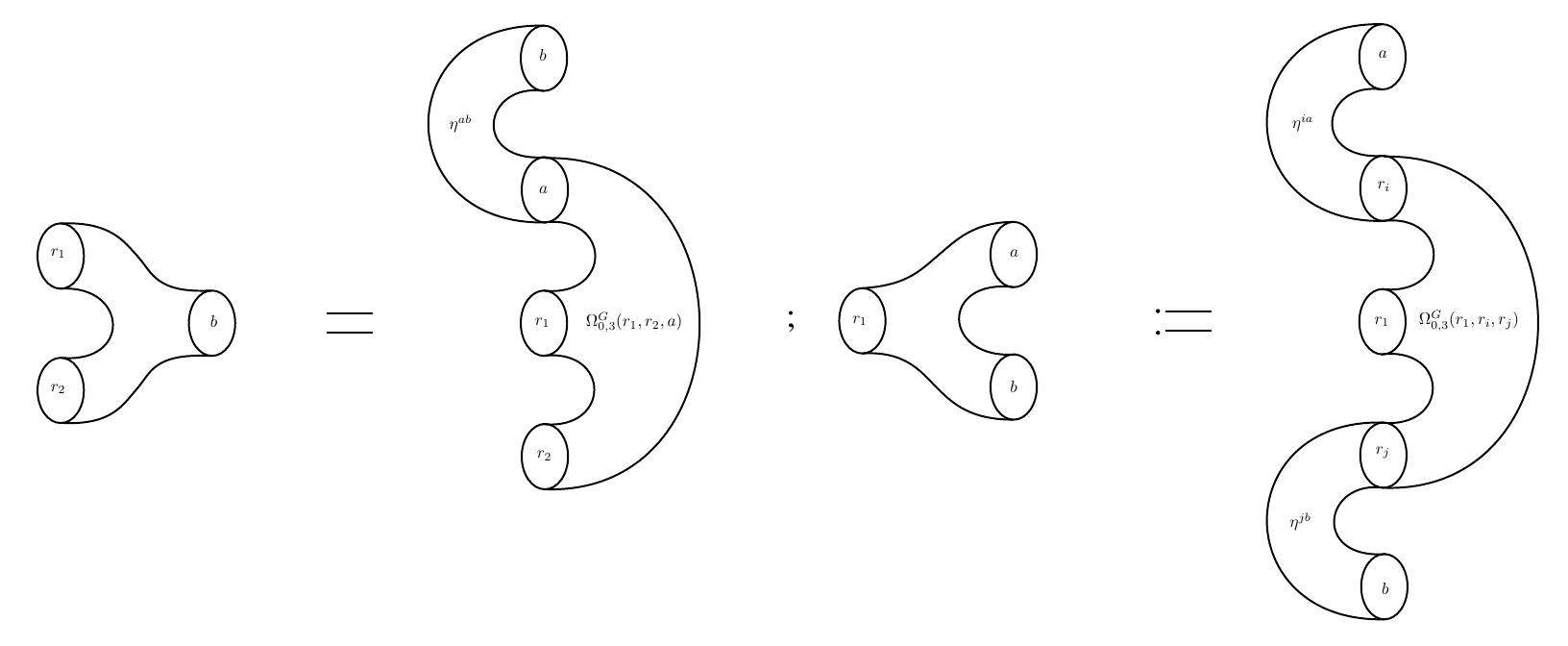, width=6.2in}
\caption{Orbifold pair of pants as product and coproduct in $A$, written in terms of the three-point function.}
\label{fig:prodPants}
\end{figure}

In this fashion, the ECA implies that $\Omega^{G}_{g,n}(v_{\cl{r_{1}}}, \dots ,v_{\cl{r_{n}}})$ satisfy the following relations
$$\Omega_{g,n}^{G}(v_{\cl{r_{1}}}, \dots ,v_{\cl{r_{n}}})=\sum_{\cl{a},\cl{b}}\Omega_{0,3}^{G}(v_{\cl{r_{1}}},v_{\cl{r_{j}}},e_{\cl{a}})\eta^{ab}\Omega_{g,n-1}^{G}(e_{\cl{b}},\bold v_{[n]\setminus\{1,j\}})$$
$$\Omega_{g,n}^{G}(v_{\cl{r_{1}}}, \dots ,v_{\cl{r_{n}}})=\sum_{\cl{r_i},\cl{r_j},\cl{a},\cl{b}}\Omega_{0,3}^{G}(v_{\cl{r_1}},e_{\cl{r_i}},e_{\cl{r_j}})\eta^{ia}\eta^{jb}\Omega_{g-1,n+1}^{G}(e_{\cl{a}},e_{\cl{b}},\bold v_{[n]\setminus\{1\}})$$
$$\Omega_{g,n}^{G}(v_{\cl{r_{1}}},\dots ,v_{\cl{r_{n}}})=\sum_{\cl{r_i},\cl{r_j},\cl{a},\cl{b}}\Omega_{0,3}^{G}(v_{\cl{r_1}},e_{\cl{r_i}},e_{\cl{r_j}})\eta^{ia}\eta^{jb}\Omega_{g_{1},|I|+1}^{G}(e_{\cl{a}},\bold v_{I})\Omega_{g_{2},|J|+1}^{G}(e_{\cl{a}},\bold v_{J})$$ where $g=g_1+g_2$ and $I \coprod J=\{2,\dots,n\}$. 

These relations are reflected by Figure \ref{fig:2} as a cutting off a pair of pants from an $n$-puntured orbifold surface. 
\begin{figure}[h!]
\epsfig{file=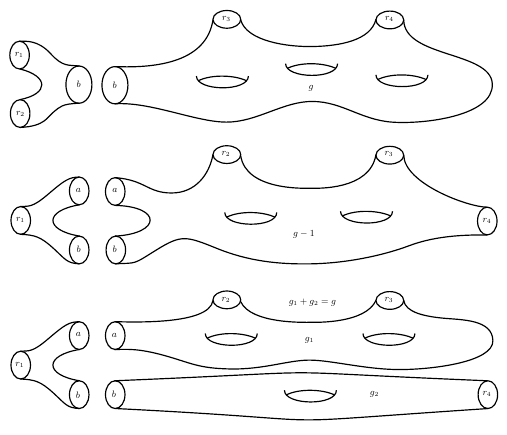, width=4.0in}
\caption{Cutting off a pair of pants from an $n$-puntured orbifold surface. Orbifold generalization of \cite[Figure 2.1]{MS}}
\label{fig:2}
\end{figure}
\qed
\end{rem}

For the $(g,n)=(1,1)$ case, let us define the $(1,1)$-operator $e\colon A \to \bC$ as the following composite:
\begin{equation}\label{e:11op}
e\colon A \overset{h}{\to} A \overset{\epsilon}{\to} \bC \,,
\end{equation}
where $h\colon A \to A$ is the handle operator of equation (\ref{e:handleop}).  The $(1,1)$-operator has one input ``decorated marking'' and no output markings and we have:
\begin{equation}\label{e:Omega11}
\Omega_{1,1}^{G}(v_{\cl{r_{1}}})=\sum_{\cl{r_i},\cl{r_j},\cl{a},\cl{b}}\Omega_{0,3}^{G}(v_{\cl{r_1}},e_{\cl{r_i}},e_{\cl{r_j}})\eta^{ia}\eta^{jb}\Omega_{0,2}^{G}(e_{\cl{a}},e_{\cl{b}})=\Omega_{0,2}^G\big (\delta(v_{\cl{r_{1}}})\big)\,,
\end{equation}

As in section \ref{ss:decdessins}, let us denote by $\cC_{g,n}(\mu_{1},\dots, \mu_{n};v_{\cl{r_{1}}}, \dots ,v_{\cl{r_{n}}})$ the number of decorated arrowed cell graphs of labeled vertices of degrees $(\mu_{1},\dots, \mu_{n})$, we shall call them in this particular case \emph{orbifold generalized Catalan numbers}. They satisfy the equation of Corollary \ref{c:twcat}. 

Applying the Laplace transform method to this equation we have by Theorem \ref{thm:main2} that the twisted meromorphic differentials 
$$\cW_{g,n}(z_1,\dots, z_n;v_{\cl{r_{1}}}, \dots ,v_{\cl{r_{n}}})=W^D_{g,n}(z_1,\dots ,z_n)\Omega^G_{g,n}(v_{\cl{r_{1}}}, \dots ,v_{\cl{r_{n}}})$$
are a solution of a twisted topological recursion (Definition \ref{d:twTR}) which splits as the product of the usual topological recursion given in \cite{DMSS} and the 2D TQFT $(A,\eta,\{\Omega^G_{g,n}\})$ given by the orbifold cohomology of $BG$ as the Frobenius algebra $A$.

\begin{rem}
The orbifold generalized Catalan numbers counts a special type of orbifold arrowed dessins. Notice that if we denote by $D$ the associated $G$-cover to a point $C\to BG$ in $\Mbar_{g,n}(BG)$, in order to define a Belyi map for $D$ compatible with the action of $G$ one has to have a commutative diagram
$$\xymatrix{
D \ar[d]_{G-equiv} \ar[r] & C\simeq D/G \ar[d]^{b} \\
\bP^{1}\ar[r] & \bP^{1}\simeq \bP^{1}/G
}$$
and since the $G$-action on $\bP^{1}$ has to fix the points $0,1$ and $\infty$, then $G$ is acting trivially on $\bP^{1}$ and therefore the quotient stack $[\bP^{1}/G]$ is isomorphic to $\bP^{1}\times BG$. Thus, we define a clean $G$-Belyi morphism of type $(g,n)$ as a pair
$$\xymatrix{C \ar[d]_{b}\ar[r]^{f}& BG\\ \bP^{1}&}$$
where $b\colon C \to \bP^{1}$ is a clean Belyi morphism and $f\colon C \to BG$ is a point in $\cM_{g,n}(BG)$. An orbifold $G$-\emph{dessin} is a representable map $b^{-1}([1,i\infty]) \to BG$.

\end{rem}

\subsection{Intersection numbers of $\Mbar_{g,n}(BG)$ and orbifold DVV equation}\quad 

Let us now briefly comment how to relate these results with the Gromov-Witten intersection numbers of $\Mbar_{g,n}(BG)$.

The \emph{$n$-point correlators} are defined in the usual way:
\begin{equation}\label{e:intG}
\langle\tau_{k_1}(e_{\cl{r_1}})\dots\tau_{k_n}(e_{\cl{r_n}})\rangle^{G}_{g,n}
:= \int_{\Mbar_{g,n}(BG)} \prod_{i=1}^{n} \bar \psi_i^{k_i} ev_i^*(e_{\cl{r_{i}}})\,,
\end{equation}
where the tautological cotangent classes $\bar \psi_i: = \varphi^*\psi_i$ are defined by pulling back the standard $\psi$-classes by the canonical projection $\varphi \colon \Mbar_{g, n}(BG) \to \Mbar_{g,n}$. In \cite{JK} it is proven that
\begin{equation}\label{e:tauG}
\langle\tau_{k_1}(e_{\cl{r_1}})\dots\tau_{k_n}(e_{\cl{r_n}})\rangle^{G}_{g,n}
= \Omega^G_g (\bold r)\langle\tau_{k_1}\dots \tau_{k_n}\rangle_{g,n}\,
\end{equation}
where $\langle\tau_{k_1}\dots \tau_{k_n}\rangle_{g,n}$ are the standard $n$-point correlators for $\Mbar_{g,n}$.

\begin{rem}
In \cite{JK} it is show that the intersection numbers on $\mgnbar(BG)$ satisfies the Virasoro constrain condition,  they use a change of basis on the Frobenius algebra to show that the result consist of $h$ copies of Witten-Kontsevich theory of a point, where $h$ is the number of conjugacy classes of $G$ (recall that $BG$ is just $h$ copies of a point). Even thought we can straightforward generalize the results of \cite{DMSS} to get $h$ copies of a Eynard-Orantin topological recursion, we shall use the approach of the previous sections: even if the intersection numbers on $\mgnbar(BG)$ are just a scalar multiple of the intersection numbers on $\mgnbar$, inside the structure there is a new theory of a CohFT-valued topological recursion we want to point out and which shall produce an orbifold generalization of the DVV equation instead of $h$ copies of the standard DVV equation.
\end{rem}

Let $N_{g,n}(\mu_1,\dots ,\mu_n)$ be the lattice point counting function of \cite{MP2012,DMSS}. By \cite[Theorem 1.3]{MP2012}, the leading terms of the Laplace transform, $F_{g,n}^{L}(t_1,\dots,t_n),$ of the number $N_{g,n}(\mu_1,\dots,\mu_n)$ of ribbon graphs of integer edge lengths $(\mu_{1},\dots, \mu_{n})$ with a cilium placed on a labeled face,
form a homogeneous polynomial of degree
$3(2g-2+n)$ given by
\begin{equation}
\label{eq:FgnK}
{F}_{g,n}^{K}(t_1,\dots,t_n)
=
\frac{(-1)^n}{2^{2g-2+n}}
\sum_{\substack{k_1+\cdots+k_n\\=3g-3+n}}
\langle\tau_{k_1}\dots\tau_{k_n}\rangle_{g,n}
\prod_{j=1}^n (2k_j-1)!! \left(
\frac{t_j}{2}
\right)^{2k_j+1},
\end{equation}

If we denote by $\cN_{g,n}(\mu_{1},\dots, \mu_{n};v_{\cl{r_{1}}}, \dots ,v_{\cl{r_{n}}})$ the number of ribbon graphs of integer edge lengths $(\mu_{1},\dots, \mu_{n})$ with a cilium placed on a labeled face and where faces are also decorated by $(v_{\cl{r_{1}}}, \dots ,v_{\cl{r_{n}}})\in A^{\otimes n}$, then this number is just the product 
$$N_{g,n}(\mu_{1},\dots, \mu_{n})\Omega^{G}_{g,n}(v_{\cl{r_{1}}}, \dots ,v_{\cl{r_{n}}})$$
and the multilinear map
$$\cN_{g,n}(\mu_{1},\dots, \mu_{n})=N_{g,n}(\mu_{1},\dots, \mu_{n})\cdot \Omega^{G}_{g,n}\colon A^{\otimes n}\to \bC$$
satisfies the following ECA based equation 
\begin{multline}
\label{eq:NG}
\mu_1  \cN_{g,n}(\mu_{1},\dots, \mu_{n};v_{\cl{r_{1}}}, \dots ,v_{\cl{r_{n}}})
=\\
=\half
\sum_{j=2} ^n 
\Bigg[
\sum_{q=0} ^{\mu_1+\mu_j}
q(\mu_1+\mu_j-q)  \cN_{g,n-1}(q,\mu_{[n]\setminus\{1,j\}};v_{\cl{r_{1}}}v_{\cl{r_{j}}},v_{\cl{r_{2}}},\dots,\widehat{v_{\cl{r_{j}}}},\dots,v_{\cl{r_{n}}})
\\
+
H(\mu_1-\mu_j)\sum_{q=0} ^{\mu_1-\mu_j}
q(\mu_1-\mu_j-q)\cN_{g,n-1}(q,\mu_{[n]\setminus\{1,j\}};v_{\cl{r_{1}}}v_{\cl{r_{j}}},v_{\cl{r_{2}}},\dots,\widehat{v_{\cl{r_{j}}}},\dots,v_{\cl{r_{n}}})
\\
-
H(\mu_j-\mu_1)\sum_{q=0} ^{\mu_j-\mu_1}
q(\mu_j-\mu_1-q)
\cN_{g,n-1}(q,\mu_{[n]\setminus\{1,j\}};v_{\cl{r_{1}}}v_{\cl{r_{j}}},v_{\cl{r_{2}}},\dots,\widehat{v_{\cl{r_{j}}}},\dots,v_{\cl{r_{n}}})
\Bigg]
\\
+\half \sum_{0\le q_1+q_2\le \mu_1}q_1q_2(\mu_1-q_1-q_2)
\Bigg[
\cN_{g-1,n+1}(q_1,q_2,\mu_{[n]\setminus\{1\}};\delta(v_{\cl{r_{1}}}),v_{\cl{r_{2}}},\dots,v_{\cl{r_{n}}})
\\
+\sum_{\substack{g_1+g_2=g\\
I\sqcup J=[n]\setminus\{1\}}} ^{\rm{stable}}
\delta^{*}\big(\cN_{g_1,|I|+1}(q_1,\mu_I;\_,v_{\cl{I}})\otimes 
\cN_{g_2,|J|+1}(q_2,\mu_J;\_,v_{\cl{J}})\big)(v_{\cl{r_1}})\Bigg].
\end{multline}
Where $
H(x) = \begin{cases}
1 \qquad x>0\\
0 \qquad x\le 0
\end{cases}
$
is the Heaviside function. This generalizes \cite[Theorem 3.3]{CMS}. The leading terms of the Laplace transform of $$\cN_{g,n}(\mu_{1},\dots, \mu_{n};v_{\cl{r_{1}}}, \dots ,v_{\cl{r_{n}}})$$ has then the shape
\begin{equation}
\label{eq:FgnKGbga}
\begin{aligned}
{\cF}_{g,n}^{K}&(t_1,\dots,t_n;v_{\cl{r_{1}}}, \dots ,v_{\cl{r_{n}}})
= \\
&=
\frac{(-1)^n}{2^{2g-2+n}}
\sum_{\substack{k_1+\cdots+k_n\\=3g-3+n}}
\langle\tau_{k_1}\dots\tau_{k_n}\rangle_{g,n}\Omega^{G}_{g,n}(v_{\cl{r_{1}}}, \dots ,v_{\cl{r_{n}}})
\prod_{j=1}^n (2k_j-1)!! \left(
\frac{t_j}{2}
\right)^{2k_j+1},
\end{aligned}
\end{equation}
where by equation (\ref{e:tauG}) we have
$$\langle\tau_{k_1}(e_{\cl{r_1}})\dots\tau_{k_n}(e_{\cl{r_n}})\rangle^{G}_{g,n}
= \langle\tau_{k_1}\dots \tau_{k_n}\rangle_{g,n}\Omega^G_g (v_{\cl{r_{1}}}, \dots ,v_{\cl{r_{n}}})\,,$$ 
Applying the Laplace transform method to ECA based equation (\ref{eq:NG}) and restricting to the top degree terms, the Frobenius algebra twisted topological recursion produces in this case the following orbifold DVV type equation
 \begin{equation}\label{e:DDV(BG)}
\begin{aligned}
& \la \tau_{k_{1}}(e_{\cl{r_1}})\cdots \tau_{k_{n}}(e_{\cl{r_n}}) \ra_{g,n}^G=\sum_{j=2}^{n} \frac{(2k_{1}+2k_{j}-1)!!}{(2k_{1}-1)!!(2k_{j}-1)!!} \\
&\times \la \tau_{k_{1}+k_{j}-1}(e_{\cl{r_{1}}}e_{\cl{r_j}})\tau_{k_{2}}(e_{\cl{r_2}})\cdots \widehat {\tau_{k_{j}}(e_{\cl{r_j}})}\cdots \tau_{k_{n}}(e_{\cl{r_n}}) \ra_{g,n}^G +\\
&+ \frac{1}{2} \sum_{l+m=k_{1}-2}\frac{(2l+1)!! (2m+1)!!}{(2k_{1}+1)!!}\sum_{\cl{r_l},\cl{r_{m}},\cl{a},\cl{b}}\phi(v_{\cl{r_1}},e_{\cl{r_l}},e_{\cl{r_{m}}})\eta^{ka}\eta^{{\ell}b}
\\
& \times \bigg( \la \tau_{l}\tau_{m}\tau_{k_{[n]\setminus \{1\}}} \ra_{g-1,n+1}\Omega^G_{g-1,n+1}(e_{\cl{a}},e_{\cl{b}},e_{\cl{\bold r_{[n]\setminus \{1\}}}})   +  \\
&+ \sum_{\substack{g_{1}+g_{2}=g\\ I\sqcup J=[n]\setminus \{1\}}}^{\text{stable}}\la \tau_l\tau_{k_I}\ra _{g_1,|I|+1}\Omega^G_{g_1,|I|+1}(e_{\cl{a}},e_{\cl{\bold r_I}}) \la \tau_m\tau_{k_J}\ra _{g_2,|J|+1}\Omega^G_{g_2,|J|+1}(e_{\cl{b}},e_{\cl{\bold r_J}}) \bigg)
\end{aligned}
\end{equation}
which we can symbolically write as
\begin{equation}\label{e:DDV(BG)simb}
\begin{aligned}
& \la \tau_{k_{1}}(e_{\cl{r_1}})\cdots \tau_{k_{n}}(e_{\cl{r_n}}) \ra_{g,n}^G=\sum_{j=2}^{n} \frac{(2k_{1}+2k_{j}-1)!!}{(2k_{1}-1)!!(2k_{j}-1)!!} \\
&\times \la \tau_{k_{1}+k_{j}-1}(e_{\cl{r_{1}}}e_{\cl{r_j}})\tau_{k_{2}}(e_{\cl{r_2}})\cdots \widehat {\tau_{k_{j}}(e_{\cl{r_j}})}\cdots \tau_{k_{n}}(e_{\cl{r_n}}) \ra_{g,n}^G +\\
&+ \frac{1}{2} \sum_{l+m=k_{1}-2}\frac{(2l+1)!! (2m+1)!!}{(2k_{1}+1)!!}
\\
& \times \bigg( \delta^{*}\big( \la \tau_{l}\tau_{m}\tau_{k_{[n]\setminus \{1\}}}(e_{\cl{\bold r_{[n]\setminus \{1\}}}}) \ra^{G}_{g-1,n+1} \big)(e_{\cl{r_{1}}})  +  \\
&+ \sum_{\substack{g_{1}+g_{2}=g\\ I\sqcup J=[n]\setminus \{1\}}}^{\text{stable}}\delta^{*}\big (\la \tau_l(\_)\tau_{k_I}(e_{\cl{\bold r_I}})\ra _{g_1,|I|+1}^G, \la \tau_m(\_)\tau_{k_J}(e_{\cl{\bold r_J}})\ra _{g_2,|J|+1}^G \big) (e_{\cl{r_{1}}})\bigg)
\end{aligned}
\end{equation}
This proves from a different perspective to that in \cite{JK} that the intersections numbers on $\Mbar_{g, n}(BG)$ satisfies the Virasoro constrain condition.

\section{Appendix: proof of Theorem \ref{thm:main2}}\label{s:proof}

We give in this section the proof of Theorem \ref{thm:main2} to make clear how to use the ECA axioms and properly compute the contributions coming from the unstable geometries. We shall reproduce the proof of \cite[Theorem 4.3]{DMSS} given in \cite[Appendix A]{DMSS} (which states the EO topological recursion satisfied by $W_{g,n}^{D}(t_1,\dots,t_n)$) and adapt that result for the twisted differentials 
$$\cW_{g,n}(t_1,\dots,t_n;v_{1},\dots ,v_{n})=W^{D}_{g,n}(t_1,\dots,t_n)\Omega_{g,n}(v_{1},\dots ,v_{n})\,.$$

Let us separate first the contributions from
unstable geometries $(g,n)=(0,1)$ and $(0,2)$ in the second term of the third line of 
equation (\ref{eq:LT of Dgn}). Using equation (\ref{ECA2-2}) of ECA2 axiom for $g_1=0$ and $I=\emptyset$, or $g_2=0$ and $J=\emptyset$ we have a contribution of 
\begin{multline}
\delta^{*}\big(\cw_{0,1}(x_1;-),
  \cw_{g,n}(x_{[n]};-,v_{[n]\setminus \{1\}})\big)(v_{1})+
 \delta^{*} \big (\cw_{g,n}(x_{[n]};-,v_{[n]\setminus \{1\}}),\cw_{0,1}(x_1;-)\big)(v_{1})\\
=2 w_{0,1}(x_1)w_{g,n}(x_1,x_2,\dots,x_n)\Omega_{g,n}(v_{1},\dots ,v_{n}) .
\end{multline}

Similarly, for $g_1=0$ and $I=\{j\}$, or $g_2=0$ and $J=\{j\}$, 
we have
\begin{multline}
\sum_{j=2}^n \delta^{*}\Big( \cw_{0,2}(x_1,x_j;-,v_{j}),
\cw_{g,n-1}(x_1,x_{_{[n]\setminus \{1,j\}}};-,v_{[n]\setminus \{1,j\}})\Big)(v_{1})=\\
=\sum_{j=2}^n w_{0,2}(x_1,x_j)w_{g,n-1}(x_1,\dots,\widehat{x_j},\dots,x_n)\Omega_{g,n}(v_{1},\dots ,v_{n}).
\end{multline}

Thus, bearing in mind that
\begin{align*}
w_{0,1}(x)&=-\frac{t+1}{t-1}\\
w_{0,2}(x_1,x_2)&=
\frac{1}{(t_1+t_2)^2}\;\frac{(t_1^2-1)^2}{8t_1}
\;\frac{(t_2^2-1)^2}{8t_2}\\
w_{g,n}(x_1,\dots,x_n)&=
(-1)^n w_{g,n}^D(t_1,\dots,t_n)\prod_{i=1}^n
\frac{(t_i^2-1)^2}{8t_i}\,,
\end{align*}
the differential equation (\ref{eq:LT of Dgn}) is equivalent to 
\begin{multline*}
 2\left(\frac{t_1^2+1}{t_1^2-1}-\frac{t_1+1}{t_1-1}
 \right)
\cw_{g,n}(t_1,\dots,t_n;v_{1},\dots ,v_{n})
\\
=
\sum_{j=2}^n
\Bigg(
\frac{(t_1^2-1)^2(t_j^2-1)^2}{16 (t_1^2-t_j^2)^2}\;
    \frac{8t_j}{(t_j^2-1)^2}\;
    \cw_{g,n-1}(t_1,\dots,\widehat{t_j},\dots,t_n;v_{1}\cdot v_{j},v_{[n]\setminus \{1,j\}})
\\
    +
    \frac{\partial}{\partial t_j}
\left(
    \frac{(t_1^2-1)(t_j^2-1)}{4 (t_1^2-t_j^2)}
     \frac{8t_1}{(t_1^2-1)^2}
     \frac{(t_j^2-1)^2}{8t_j}
         \cw_{g,n-1}(t_2,\dots,t_n;v_{1}\cdot v_{j},v_{[n]\setminus \{1,j\}})
         \right)
     \Bigg)
\\
+
\frac{(t_1^2-1)^2}{8t_1}
\bigg(
\delta^{*}\big(\cw_{g-1,n+1}(t_1,t_1,t_2,\dots,t_n;-,-,v_{[n]\setminus \{1\}})\big)(v_{1}) \\
+\sum_{\substack{g_1+g_2=g\\
I\sqcup J=\{2,\dots,n\}}} ^{\text{stable}}
\delta^{*}\big(\cw_{g_1,|I|+1}(t_1,t_I;-,v_{I}),
\cw_{g_2,|J|+1}(t_1,t_J;-,v_{J})\big)(v_{1})
\bigg)
\\
+
2\sum_{j=2}^n
\frac{1}{(t_1+t_j)^2}\;\frac{(t_1^2-1)^2}{8t_1}
w_{g,n-1}^D(t_1,\dots,\widehat{t_j},\dots,t_n)\Omega_{g,n}(v_{1},\dots ,v_{n})
\\
=
\sum_{j=2}^n
\Bigg(
\left(
\frac{ t_j(t_1^2-1)^2}{2 (t_1^2-t_j^2)^2}
     +  
\frac{1}{(t_1+t_j)^2}\;\frac{(t_1^2-1)^2}{4t_1}
\right)
\cw_{g,n-1}(t_1,\dots,\widehat{t_j},\dots,t_n;v_{1}\cdot v_{j},v_{[n]\setminus \{1,j\}})
\\
    +
    \frac{t_1}{t_1^2-1}\;
    \frac{\partial}{\partial t_j}
\left(
    \frac{(t_j^2-1)^3}{4t_j (t_1^2-t_j^2)}\;
         \cw_{g,n-1}(t_2,\dots,t_n;v_{1}\cdot v_{j},v_{[n]\setminus \{1,j\}})
         \right)
     \Bigg)
     \\
+
\frac{(t_1^2-1)^2}{8t_1}
\bigg(
\delta^{*}\big(\cw_{g-1,n+1}(t_1,t_1,t_2,\dots,t_n;-,-,v_{[n]\setminus \{1\}})\big)(v_{1})\\
+\sum_{\substack{g_1+g_2=g\\
I\sqcup J=\{2,\dots,n\}}} ^{\text{stable}}
\delta^{*}\big(\cw_{g_1,|I|+1}(t_1,t_I;-,v_{I}),
\cw_{g_2,|J|+1}(t_1,t_J;-,v_{J})\big)(v_{1})
\bigg).
\end{multline*}
Where we have used ECA1 equation (\ref{ECA1}) to convert the second and sixth lines into the seventh line.

Since
$$
 2\left(\frac{t_1^2+1}{t_1^2-1}-\frac{t_1+1}{t_1-1}
 \right)
=
-\frac{4t_1}{t_1^2-1},
$$
we obtain
\begin{multline}
\label{eq:D-diff}
\cw_{g,n}(t_1,\dots,t_n;v_{1},\dots ,v_{n})
=\\
=-\sum_{j=2}^n
\Bigg(
    \frac{\partial}{\partial t_j}
\left(
    \frac{(t_j^2-1)^3}{16t_j (t_1^2-t_j^2)}\;
         \cw_{g,n-1}(t_2,\dots,t_n;v_{1}\cdot v_{j},v_{[n]\setminus \{1,j\}})
         \right)
     \\
      +
  \frac{(t_1^2-1)^3}  {16 t_1^2}
  \;
  \frac{t_1^2+t_j^2}{(t_1^2-t_j^2)^2}
\cw_{g,n-1}(t_1,\dots,\widehat{t_j},\dots,t_n;v_{1}\cdot v_{j},v_{[n]\setminus \{1,j\}})
     \Bigg)
     \\
- \frac{(t_1^2-1)^3}{32 t_1^2}
\bigg( \delta^{*}\big (\cw_{g-1,n+1}(t_1,t_1,t_2,\dots,t_n;-,-,v_{[n]\setminus \{1\}})\big)(v_{1})
\\
+\sum_{\substack{g_1+g_2=g\\
I\sqcup J=\{2,\dots,n\}}} ^{\text{stable}}
\delta^{*}\big(\cw_{g_1,|I|+1}(t_1,t_I;-,v_{I}),
\cw_{g_2,|J|+1}(t_1,t_J;-,v_{J})\big)(v_{1})
\bigg).
\end{multline}

Now let us compute the integral (notice that we are already separating the unstable $(0,2)$ differential forms and substituting the kernel by its value (\ref{eq:Gkernel})). 

\begin{multline}
\label{eq:DEO-A}
\cW_{g,n}(t_1,\dots,t_n;v_{1},\dots ,v_{n})
=
-\frac{1}{64} \; 
\frac{1}{2\pi i}\int_\phi
\left(
\frac{1}{t+t_1}+\frac{1}{t-t_1}
\right)
\frac{(t^2-1)^3}{t^2}\cdot \frac{1}{dt}\cdot dt_1
\\
\times
\Bigg[
\sum_{j=2}^n
\bigg(
\delta^{*}\big(\cW_{0,2}(t,t_j;-,v_{j}),\cW_{g,n-1}(-t,t_{[n]\setminus \{1,j\}};-,v_{[n]\setminus \{1,j\}})\big)(v_{1})
+\\
\delta^{*}\big(\cW_{0,2}(-t,t_j;-,v_{j}),\cW_{g,n-1}(t,t_{[n]\setminus \{1,j\}};-,v_{[n]\setminus \{1,j\}})\big)(v_{1})
\bigg)
\\
+
\delta^{*}\big(\cW_{g-1,n+1}(t,{-t},t_2,\dots,t_n;-,-,v_{[n]\setminus \{1\}})\big)(v_{1})\\
+
\sum^{\text{stable}} _
{\substack{g_1+g_2=g\\I\sqcup J=\{2,3,\dots,n\}}}
\delta^{*}\big(\cW_{g_1,|I|+1}(t,t_I;-,v_{I}),\cW_{g_2,|J|+1}({-t},t_J;-,v_{J})\big)(v_{1})
\Bigg].
\end{multline}
Recall that for $2g-2+n>0$, $w_{g,n}^D(t_1,\dots,t_n)$
is a Laurent polynomial in $t_1^2,\dots,t_n^2$.
Thus the last two lines of (\ref{eq:DEO-A}) are immediately calculated because the integration contour $\phi$ encloses $\pm t_1$ and  contributes residues with the negative sign.
 The result is exactly the last two lines of (\ref{eq:D-diff}).
Similarly, since
\begin{multline*}
\delta^{*}\big(\cW_{0,2}(t,t_j;-,v_{j}),\cW_{g,n-1}(-t,t_{[n]\setminus \{1,j\}};-,v_{[n]\setminus \{1,j\}})\big)(v_{1})
+\\
+\delta^{*}\big(\cW_{0,2}(-t,t_j;-,v_{j}),\cW_{g,n-1}(t,t_{[n]\setminus \{1,j\}};-,v_{[n]\setminus \{1,j\}})\big)(v_{1})=\\
=
-\left(
\frac{1}{(t+t_j)^2}+\frac{1}{(t-t_j)^2}
\right)\Omega_{g,n}(v_{1},\dots ,v_{n})
w_{g,n-1}^D(t,t_2,\dots,\widehat{t_j},\dots,t_n)
\;dt\;dt\;dt_2\cdots\widehat{dt_j}\cdots dt_n,
\end{multline*}
the residues at $\pm t_1$ contributes 
$$
-\frac{(t_1^2-1)^3(t_1^2+t_j^2)}
{16 t_1^2(t_1^2-t_j^2)^2}\;
w_{g,n-1}^D(t_1,\dots,\widehat{t_j},\dots,t_n)\Omega_{g,n}(v_{1},\dots ,v_{n}).
$$
which, by ECA1 axiom of equation (\ref{ECA1}) is equal to:
$$
-\frac{(t_1^2-1)^3(t_1^2+t_j^2)}
{16 t_1^2(t_1^2-t_j^2)^2}\;
w_{g,n-1}^D(t_1,\dots,\widehat{t_j},\dots,t_n)\Omega_{g,n-1}(v_{1}\cdot v_{j},v_{[n]\setminus \{1,j\}}).
$$
This is the same as the second line of the 
right-hand side of (\ref{eq:D-diff}).

Within the contour $\gam$, there are second order poles at $\pm t_j$ for each $j\ge 2$ which come from $(0,2)$ unstable cases, using ECA 1 and ECA 2 axioms of equations (\ref{ECA1}) and (\ref{ECA2-2}) respectively, we calculate
\begin{multline*}
\frac{1}{64} \; 
\frac{1}{2\pi i}\int_\gam
\left(
\frac{1}{t+t_1}+\frac{1}{t-t_1}
\right)
\frac{(t^2-1)^3}{t^2}\\
\sum_{j=2}^n
\bigg(
\delta^{*}\big(\cw_{0,2}(t,t_j;-,v_{j}),\cw_{g,n-1}(-t,t_{[n]\setminus \{1,j\}};-,v_{[n]\setminus \{1,j\}})\big)(v_{1})+
\\
+
\delta^{*}\big(\cw_{0,2}(-t,t_j;-,v_{j}),\cw_{g,n-1}(t,t_{[n]\setminus \{1,j\}};-,v_{[n]\setminus \{1,j\}})\big)(v_{1})
\bigg)
\\
=
-\frac{1}{32} 
\frac{\partial}{\partial t_j}
\left(
\left(
\frac{1}{t_j+t_1}+\frac{1}{t_j-t_1}
\right)
\frac{(t_j^2-1)^3}{t_j^2}
w_{g,n-1}^D(t_j,t_2,\dots,\widehat{t_j},\dots,t_n)\Omega_{g,n}(v_{1},\dots ,v_{n})
\right)
\\
=
-\frac{1}{16}
\frac{\partial}{\partial t_j}
\left(
\frac{1}{t_j^2-t_1^2}\;
\frac{(t_j^2-1)^3}{t_j}
w_{g,n-1}^D(t_j,t_2,\dots,\widehat{t_j},\dots,t_n)\Omega_{g,n-1}(v_{1}\cdot v_{j},v_{[n]\setminus \{1,j\}})
\right).
\end{multline*}
This gives the first line of the right-hand side
of (\ref{eq:D-diff}).
We have thus
 completed the proof of Theorem \ref{thm:main2}.

\begin{ack}
The author wish to thank Olivia Dumitrescu, Motohico Mulase and Takashi Kimura for useful and valuable discussions.
\end{ack}


\providecommand{\bysame}{\leavevmode\hbox to3em{\hrulefill}\thinspace}

\bibliographystyle{amsplain}

\begin{thebibliography}{10}


\bibitem{ACV}
D.~Abramovich, A.~Corti,  A.~ Vistoli,
{\em Twisted bundles and admissible covers}, Comm. Algebra {\bf 31}, no. 8, 3547--3618  (2003). 

\bibitem{AV}
D.~ Abramovich, A.~ Vistoli,
{\em Compactifying the space of stable maps}, J. Amer. Math. Soc. {\bf 15} (2002)
27--75.


\bibitem{ABO} J.E.~Andersen, G. ~Borot, N.~OrantinÊ
\emph{Modular functors, cohomological field theories, and topological recursion}, arXiv:1509.01387v3 (2016)

\bibitem{At}
M.F. ~Atiyah
\emph{Topological quantum field theories},  Inst. Hautes ƒtudes Sci. Publ. Math. {\bf 68}, 175Ð186 (1988). 

\bibitem{BHLM}
V.~Bouchard, D.~Hern\'andez Serrano, X..~Liu, and M.~Mulase,
\emph{Mirror Symmetry for orbifold Hurwitz numbers},
Journal of Differential Geometry.
 \textbf{98.3}, 375-.423 (2014).
%
\bibitem{BKMP}
V.~Bouchard, A.~Klemm, M.~Mari\~no, and S.~Pasquetti,
\emph{Remodeling the {B}-model},
Commun.\ Math.\ Phys.
 \textbf{287}, 117--178 (2008).



\bibitem{CMS}
K.~Chapman, M.~Mulase, and B.~Safnuk,
\emph{
Topological recursion and the Kontsevich constants
for the volume of the moduli of curves}, Commun. Number Theory Phys. \textbf{5} (2011), no.3, 643-698.

\bibitem{C} L.~Chen,
\emph{Bouchard-Klemm-Marino-Pasquetti Conjecture for $\bC^3$}, arXiv:0910.3739 (2009)

\bibitem{CR}
W.~Chen, Y.~Ruan,
{\em Orbifold Gromov-Witten theory}, in Orbifolds in mathematics
and physics (Madison, WI, 2001), Contemp. Math. vol \textbf{310},  25-85. Amer.
Math. Soc., Providence, RI, 2002 [arXiv:math.AG/0103156].

 
\bibitem{OM3}
O.~Dumitrescu and M.~Mulase,
\emph{Edge-contraction on dual ribbon graphs, 
2D TQFT, and the mirror of orbifold Hurwitz 
numbers}, 
arXiv:1508.05922 (2015).

\bibitem{DMSS}
O.~Dumitrescu, M.~Mulase,  A.~Sorkin, 
and B.~Safnuk,
\emph{The Laplace transform, mirror symmetry, and 
the Eynard-Orantin topological recursion}, in Algebraic and geometric aspects of integrable systems and random matrices, 263-315, Contemp. Math., \textbf{593}, Amer. Math. Soc., Providence, RI, 2013.


\bibitem{DOSS}
P.~Dunin-Barkowski, N.~Orantin, S.~Shadrin, and L.~Spitz,
{\em Identification of the Givental formula with the spectral curve topological recursion procedure}, Commun. Math. Phys. \textbf{328}, 669-700 (2014).



  \bibitem{EMS} B.~Eynard, M.~Mulase
  and B.~Safnuk,
  \emph{The
{L}aplace transform  of the cut-and-join equation
and the {B}ouchard-{M}ari\~no conjecture on
{H}urwitz numbers}, Publ. Res. Inst. Math. Sci. \textbf{47} (2011), no. 2, 629-670.


\bibitem{EO1}
  B.~Eynard and N.~Orantin,
{\em Invariants of algebraic curves and topological expansion},
Communications in Number Theory
and Physics {\bf 1},  347--452 (2007).


\bibitem{EO2} B.~Eynard and N.~Orantin, 
{\em Weil-Petersson volume of moduli spaces, Mirzakhani's recursion and matrix models},
  arXiv:0705.3600 [math-ph] (2007).
  
\bibitem{EO3} B.~Eynard and N.~Orantin, 
{\em Computation of open Gromov-Witten invariants for toric Calabi-Yau 3-folds by topological recursion, a proof of the BKMP conjecture}, arXiv:1205.1103v2 [math-ph] (2013).

\bibitem{FLZ}
B.~Fang, C.-C. M.~Liu, and Z.~Zong,
{\em All genus open-closed mirror symmetry for affine toric Calabi-Yau 3-orbifolds}, arXiv:1310.4818 [math.AG]. (2013).  To appear in Proc. Of Symposioa Pure Math.
 
\bibitem{JK}
T.J.~Jarvis and T.~Kimura
\emph{Orbifold quantum cohomology of the 
classifying space of a finite group},
In ``Orbifolds in mathematics and physics,''
Contemp.\ Math.\ \textbf{310}, 123--134 (2002).


\bibitem{Kock}
J.~Kock
\emph{Frobenius algebras and 2D topological quantum field theories}, London Mathematical Society Student Texts \textbf{59}. Cambridge University Press, Cambridge (2004).


\bibitem{LX2} K. Liu and H. Xu, {\em Recursion formulae
of  Higher Weil--Petersson volumes}, 
Intern.\ Math.\ Res.\ Notices, (2009), No. 5,  
835--859 (2009).


\bibitem{Mir1} M.~Mirzakhani, {\em Simple geodesics and Weil-Petersson volumes of moduli spaces of bordered Riemann surfaces},
Invent. Math. {\bf 167}, 179--222 (2007).

\bibitem{Mir2} M.~Mirzakhani, {\em Weil-Petersson volumes and intersection theory on the moduli space
of curves}, J. Amer. Math. Soc. {\bf 20}, 1--23 (2007).



\bibitem{MP1998}
 M.~Mulase and M.~Penkava, \emph{
 Ribbon graphs, quadratic differentials on Riemann surfaces, and algebraic curves defined over $\overline{\mathbb{Q}}$}, The Asian Journal of Mathematics  \textbf{2} (4), 875--920 (1998).


   \bibitem{MP2012}
 M.~Mulase and M.~Penkava, 
 \emph{Topological recursion for the Poincar\'e 
 polynomial of the combinatorial moduli space of curves}, Adv. Math. {\bf 230} (2012), no. 3, 1322-1339.
 


\bibitem{MS} M. Mulase and B. Safnuk, {\em Mirzakhani's recursion relations, Virasoro constraints and the KdV hierarchy},
Indian J. Math. \textbf{50}, 189--228 (2008).

   \bibitem{MSu}
 M.~Mulase and P.~Sulkowski, 
 \emph{Spectral curves and the Schroedinger equations for the Eynard-Orantin recursion}, 
 Preprint  arXiv:1210.3006v3 math-ph (2012).

\bibitem{MZ} M.~Mulase and N.~Zhang,
\emph{Polynomial recursion formula for linear Hodge integrals},  Commun. Number Theory Phys., Vol.{\bf 4} (2010), No.2, 267--294.


  \bibitem{N1} P.~Norbury,
  \emph{Counting lattice points in the moduli space of curves},
  Mathematical research letters, Vol. {\bf 17} (2010), Nº 2-3, 467-482.


  \bibitem{N2} P.~Norbury,
  \emph{String and dilaton equations for counting lattice points in the moduli space of curves},  Transactions of the American Mathematical Society, Vol. {\bf 365} (2013), Nº 4,1687-1709.

\bibitem{Zhou2} J.~Zhou,
\emph{Local Mirror Symmetry for One-Legged Topological Vertex}, arXiv:0910.4320 (2009)


\end{thebibliography}

\end{document}